\documentclass[10pt]{amsart} 
\usepackage{soul}
\usepackage{cancel}
\usepackage{tikz}
\usepackage{float}
\usepackage{latexsym,amsfonts,amssymb,amsmath,amsthm}
\usepackage[nameinlink]{cleveref}
\usepackage{etaremune}
\usepackage{pifont}
\usepackage{rotating}

\tikzset{cross/.style={path picture={
  \draw
    (path picture bounding box.south east)--(path picture bounding box.north west)
    (path picture bounding box.south west)--(path picture bounding box.north east);}}}
\newtheorem{theorem}{Theorem} 
\newtheorem{corollary}{Corollary} 
\newtheorem{proposition}{Proposition} 
\newtheorem{lemma}{Lemma} 
\newtheorem{definition}{Definition} 
\newtheorem{rmk}{Remark} 
\newtheorem{ex}{Example} 
\newtheorem{question}{Question} 
\newtheorem{conjecture}{Conjecture} 
\DeclareMathOperator{\HS}{HS}
\DeclareMathOperator{\HF}{HF}

\DeclareMathOperator{\ini}{in}

\usepackage{color}

\begin{document}
\title{On decomposing 
monomial algebras with the Lefschetz properties}

\author{Oleksandra Gasanova
  }\address{Department of Mathematics, Uppsala University, SE-751 06 Uppsala, Sweden}\email{oleksandra.gasanova@math.uu.se} 

\author{Samuel
  Lundqvist}\address{Department of Mathematics, Stockholm University, SE-106 91 Stockholm, Sweden}\email{samuel@math.su.se} 
  
  \author{Lisa
  Nicklasson}\address{Dipartimento di Matematica, Università degli Studi di Genova, Via Dodecaneso 35, 16146 Genova, Italy}\email{nicklasson@dima.unige.it}


  

  
\begin{abstract}
 We introduce a general technique for decomposing monomial algebras which 
we use to study the Lefschetz properties. 
In particular, we prove that Gorenstein codimension three algebras arising from numerical semigroups have the strong Lefschetz property, and we give partial results on monomial almost complete intersections.
We also study the reverse of the decomposition process -- a gluing operation -- which gives a way to construct monomial algebras with the Lefschetz properties.
\end{abstract}
\maketitle

\section{Introduction}

The Lefschetz properties of graded Artinian algebras is an active area of research in commutative algebra with connections to many branches of mathematics, including algebraic geometry, combinatorics, topology, and representation theory, see \cite{book} and \cite{atour} for an overview. Recall that an Artinian standard graded algebra $A$ has the WLP if there is a linear form $\ell$ such that the map induced by multiplication by $\ell$ has maximal rank 
 in every degree, while $A$ has the SLP if the map induced by $\ell^i$ 
 has maximal rank
 in every degree, for all $i$. Although the research community is making steady progress on the classification of graded Artinian algebras with respect to the Lefschetz properties, there are many natural families of algebras which still remain uncategorized. 

In this paper we study some of these classes, including codimension three Gorenstein algebras and equigenerated monomial almost complete intersections.
To perform most of these studies, we introduce a method to decide whether a monomial algebra has one of the Lefschetz properties by splitting its $\Bbbk$-basis into smaller pieces.

Even though our technique is only applicable for monomial algebras, we can 
derive results 
also on non-monomial algebras using a result by Wiebe \cite{ini}: if $\Bbbk[x_1,\ldots,x_n]/\ini_{\prec}(I)$ has the weak/strong Lefschetz property (WLP/SLP), then so does  $\Bbbk[x_1,\ldots,x_n]/I$, giving the connection to the monomial case.
Here $\ini_{\prec}(I)$ denotes the initial ideal with respect to the monomial order $\prec$. 

The class of Gorenstein algebras that we consider may be seen as a generalization of the
Gorenstein codimension three algebras associated to the Ap\'ery set of a numerical semigroup described by   Guerrieri \cite{guerrieri}. We show that algebras in this class have the SLP, which extends and simplifies recent results by Mir\'{o}-Roig and Tran on the presence of the WLP for a subclass of the Ap\'ery set class described above \cite{gorensteinnumerical,gorensteinnumerical2}.
Our class of Gorenstein algebras with the SLP also covers a class described by Iarrobino, McDaniel, and Seceleanu  \cite[Corollary 5.14] {gorensteinglue}.

In the monomial almost complete intersection part, we restrict ourselves to the equigenerated case. We 
show that $\Bbbk[x_1,\ldots,x_n]/(x_1^d,\ldots,x_n^d,x_1^{d-1} x_{2})$, $\Bbbk$ a field of characteristic zero, has the SLP for all $n$ and conjecture necessary and sufficient conditions for the WLP and the SLP in codimension three. Moreover, we prove the necessary part in the WLP case and, using a recent result by Cook II and Nagel \cite{maci}, provide a reduction for the sufficient part.

We also give examples in the literature which can be seen as special cases of our method, and which inspired us to investigate the basis splitting argument; a result on the WLP for powers of monomial complete intersections by Boij, Fr\"oberg and the second author \cite{powers}, and a result on the SLP of complete intersections of products of linear forms by Juhnke-Kubitzke, Mir\'o-Roig, Murai, and Wachi  \cite{productlinear}.

It is our hope that the combination of Wiebe's result with our decomposition method will provide a useful technique for  establishing the Lefschetz properties also for other classes of graded Artinian algebras than we consider in this paper. 

We organize the paper as follows.
In \Cref{sec:splitting} we introduce our basis splitting argument, which gives a way of decomposing and gluing algebras. 
 In \Cref{sec:gor} we apply the decomposition method and give a class of Gorenstein algebras with the SLP. In \Cref{sec:table} we introduce the class of \emph{ideals coming from tables}, extending our result in \Cref{sec:gor}.
In \Cref{sec:maci} we decompose 
equigenerated monomial almost complete intersection algebras, to give results on the Lefschetz properties of these algebras. 
The last part, \Cref{sec:gluing}, concerns the gluing operation, which can be used to construct monomial algebras with the Lefschetz properties.

Finally, some words on the notation.
For the rest of the paper, 
$\Bbbk$ will denote a field of characteristic zero, and $R$ will denote the standard graded polynomial algebra in $n$ variables over $\Bbbk$. If $I$ is a monomial ideal in $R$, we let $I^c$ denote the set of monomials in $R$ which are not in $I$. For any $f \in R$ and any subset $X$ of $\Bbbk[x_1,\ldots,x_n]$, we let $fX=\{fg \ | \ g \in X\}$. For a graded algebra $S=R/I$, $I$ a homogeneous ideal, 
we let $\HF(S,i)$ denote the $\Bbbk$-dimension of the $i$th graded component $S_i$. Also, we let 
$\HS(S,t)$ denote the Hilbert series of $S$; $\HS(S,t) = \sum_i \HF(S,i) t^i$.

\section{A basis splitting argument for monomial algebras} \label{sec:splitting}

We begin by a well known result, which we state as a lemma. 
\begin{lemma}\label{lemma:glue}
Let $K$ be a monomial ideal in $R$ and let $m$ be a monomial. Let $I=K+(m)$ and $J=K:(m)$ Then 
$$K^c = I^c \sqcup mJ^c, $$

and it follows that 
\[
\HS(R/K,t) = \HS(R/I,t) +t^{\deg(m)}\HS(R/J,t).\]
\end{lemma}

\begin{proof}
Monomials in $K^c$ which are not divisible by $m$ are exactly monomials in $(K+(m))^c=I^c$.
Monomials in $K^c$ which are divisible by $m$ are multiples of $m$ by monomials in $(K:(m))^c=J^c.$
\end{proof}


If $I$, $J$, $K$, and $m$ are as in Lemma \ref{lemma:glue} we say that $K$ can be \emph{decomposed} into $I$ and $J$ w.\,r.\,t.\ $m$. We may also think of $K$ as a \emph{gluing} of $I$ and $J$ w.\,r.\,t.\ $m$, an operation discussed further in \Cref{sec:gluing}.

Before we state and prove our main decomposition result we recall that when $A$ has the WLP (SLP), the element $\ell$ for which the multiplication maps have maximal rank is called a weak (strong) Lefschetz element. If $A$ is a monomial algebra, $A$ has the WLP (SLP) if and only if the sum of the variables is a weak (strong) Lefschetz element \cite{MMN}. 

\begin{theorem}\label{thm:glue_maxrank}
Let $K$, $m$, $I$, $J$ be as in \Cref{lemma:glue}, and in addition assume 
\[\HF(R/I,i)\!<\!\HF(R/I,i+d) \implies \HF(R/J,i-\deg(m))\! \le\! \HF(R/J,i-\deg(m)+d)\]
and
\[\HF(R/I,i)\!>\!\HF(R/I,i+d) \implies \HF(R/J,i-\deg(m))\! \ge \! \HF(R/J,i-\deg(m)+d)\]
for all $i$ and some positive integer $d$. Let $f$ be a form of degree $d$. If multiplication by $f$ has maximal rank in every degree in $R/I$ and $R/J$ then it does so in $R/K$. 

In particular, if the above assumptions on the Hilbert functions hold for $d=1$ (respectively, for all $d\ge 1$), then $R/K$ has the WLP (respectively, SLP) if both $R/I$ and $R/J$ do.
\end{theorem}
\begin{proof}
	By \Cref{lemma:glue} we have $R/K=R/I \oplus m(R/J)$, considered as vector spaces. With this decomposition of $R/K$ we can represent the multiplication map $\cdot f: [R/K]_i \to [R/K]_{i+d}$ by a block matrix
	\[
	M=\begin{pmatrix}
	\cdot f :[R/I]_i \to [R/I]_{i+d} & 0 \\
	* & \cdot f : [R/J]_{i-\deg(m)} \to [R/J]_{i-\deg(m)+d}
	\end{pmatrix}
	\]
	with the convention $[R/J]_j=\{0\}$ if $j<0$. The diagonal blocks correspond to multiplication by $f$ in $R/I$ and $R/J$, and it is a lower triangular block matrix since $f(m(R/J)) \subseteq m(R/J)$. By the condition given on the Hilbert functions it can not happen that one of the diagonal blocks has strictly more rows than columns while the other block has strictly fewer rows than columns. Hence $M$ has maximal rank if the two diagonal blocks have maximal ranks. This is the same as saying that multiplication by $f$ has maximal rank in every degree in $R/K$ if it does so in both $R/I$ and $R/J$. 		
\end{proof}

From now on we will only consider Artinian algebras. For an Artinian algebra $A=\bigoplus_i A_i$, the largest $s$ such that $A_s\not=0$, is called the \emph{maximal socle degree} of $A$. If $\HF(A,i)=\HF(A,s-i)$ for all $i\in \{0,\ldots, s\}$, we say that $A$ has a symmetric Hilbert series. We say that 
$A$ has the \emph{SLP in the narrow sense} if $A$ has the SLP and its Hilbert series is symmetric.

\begin{corollary}\label{cor:symmetricHS} 
Let $K$ be an Artinian monomial ideal and let $m$ be a monomial. Let $I=K+(m)$ and $J=K:(m)$, and note that these ideals are Artinian as well. Assume that
\begin{enumerate}
\item $R/I$ has a symmetric Hilbert series and the maximal socle degree $r$,
\item $R/J$ has a symmetric Hilbert series and the maximal socle degree $s$ such that $r-s=2\deg(m)$.
\end{enumerate}
Then $R/K$ has a symmetric Hilbert series and the maximal socle degree $r$. Moreover, if $R/I$ and $R/J$ both have the WLP (SLP) then so does $R/K$. 
\end{corollary}

\begin{proof} 
As the series $t^{\deg(m)}\HS(R/J,t)$ is symmetric around $s/2+\deg(m)=r/2$, it follows from \Cref{lemma:glue} that the Hilbert series of $R/K$ is symmetric. Since $\HS(R/K,t)=t^{\deg(m)}\HS(R/J,t)+\HS(R/I,t)$, the maximal socle degree equals $\max(r,s+\deg(m))=r$.

Now assume that $R/I$ and $R/J$ both have the WLP (SLP). Then they both have unimodal Hilbert series. Since $R/I$ has a symmetric Hilbert series and the maximal socle degree $r$, the inequality $\HF(R/I,i)<\HF(R/I,i+d)$ implies that $2i+d<r$. As $r=s+2\deg(m)$, this is equivalent to $2(i-\deg(m))+d<s$. This, in turn, implies $\HF(R/J,i-\deg(m)) \le \HF(R/J,i-\deg(m)+d)$, since $R/J$ has symmetric Hilbert series and the maximal socle degree $s$. In the same way $\HF(R/I,i)>\HF(R/I,i+d)$ implies $2i+d>r$, which then implies $\HF(R/J,i-\deg(m)) \ge \HF(R/J,i-\deg(m)+d)$. It now follows from \Cref{thm:glue_maxrank}, with $f$ being the sum of the variables, that $R/K$ has the WLP (SLP).
\end{proof}

When $K$ is glued symmetrically from $I$ and $J$ w.\,r.\,t.\ $m$ as in \Cref{cor:symmetricHS} we say that the gluing is \emph{centre-to-centre}.

\begin{ex}
\label{ex:1}
	In \cite{productlinear} it is proved that algebras in a certain class of complete intersections generated by products of linear forms have the SLP. The key to the result is that one can determine the associated 
	initial ideals, and prove that these monomial ideals have the SLP. The initial ideals are of the form
	\[K=(x_1^{d_1+1},\ldots, x_{n-1}^{d_{n-1}+1}, x_1x_n^{d_0}, x_2x_n^{d_0+d_1}, \ldots, x_nx_n^{d_0+\dots +d_{n-1}}),\]
	where $d_0, d_1, \ldots, d_{n-1}$ are positive integers.
	We will now give a proof of the fact that this ideal defines an algebra with the SLP in the narrow sense and the maximal socle degree $d_{0}+ \cdots + d_{n-1}$, using \Cref{cor:symmetricHS}.
	
	The proof is by induction over $n$. For $n=1$ the statement is true because all algebras in one variable have the SLP in the narrow sense.
	
	Let $I=K+(x_1)$ and 
	\[J=K:(x_1)=  (x_1^{d_1},x_2^{d_2+1},\ldots,x_{n-1}^{d_{n-1}+1},x_n^{d_0}). \]
Notice that 
\[
\frac{R}{I}\cong \frac{\Bbbk[x_2, \ldots, x_n]}{(x_2^{d_2+1},\ldots, x_{n-1}^{d_{n-1}+1}, x_2x_n^{d_0+d_1}, \ldots, x_nx_n^{d_0+\dots +d_{n-1}})}
\]	
	and the latter algebra
	 has the SLP in the narrow sense with maximal socle degree $r=d_0+ \cdots + d_{n-1}$, by the inductive assumption. Also $R/J$ has the SLP in the narrow sense as it is a monomial complete intersection \cite{stanley}. The maximal socle degree of $R/J$ is $s=d_{0}+ \cdots + d_{n-1} -2$. As $r-s=2$ we can use \Cref{cor:symmetricHS} with $m=x_1$, and it follows that $R/K$ has the SLP in the narrow sense.  
	
\end{ex}

\begin{ex}
In \cite{powers} it is proved that $\Bbbk[x_1,\ldots,x_n]/(x_1^2,\ldots,x_n^2)^2$ has the WLP when $n$ is odd. We will sketch the proof in terms of \Cref{thm:glue_maxrank} and will do so by defining a sequence of algebras and then arguing that they all have the WLP. Let $K_0 := (x_1^2,\ldots,x_n^2)^2$. 
Let $K_1:= K_0 + (x_1^2) = x_1^2  + (x_2^2,\ldots,x_n^2)^2$ and let
$J_1 := K_0:(x_1^2) = (x_1^2, \ldots,x_n^2)$.
Next, let 
$K_2 := K_1  + (x_2^2)  = (x_1^2) + (x_2^2) + (x_3^2,\ldots,x_n^2)^2$ 
and 
$J_2 :=  K_1 : (x_2^2) = (x_1^2, \ldots,x_n^2).$ Continue this way to define ideals $K_1,\ldots,K_{n-1}, J_1,\ldots,J_{n-1}$, with $K_i = K_{i-1} + (x_i^2)$ and $J_i :=  K_{i-1} : (x_i^2)$. In particular, 
$K_{n-1} = (x_1^2,\ldots,x_{n-1}^2,x_n^4)$ and $J_1= \cdots = J_{n-1} = (x_1^2,\ldots,x_n^2)$, so being monomial complete intersections, $K_{n-1}$ and $J_1, \ldots, J_{n-1}$ have the SLP. By \Cref{lemma:glue}, we have that $K_{i-1}^c = K_{i}^c + x_{i}^2 J_{i}^c$ for $i = 1, \ldots, n-1$. One can check that 
$K_{i-1}, (x_i^2), K_i, J_i$ fulfill the requirement in \Cref{thm:glue_maxrank} for $i =n-1, \ldots, 1$, with $d=1$ and $n$ odd, which involves some work with binomial coefficients, see \cite{powers}. This gives that $K_{n-2}, \ldots, K_0$ define algebras with the WLP. 

Remark: The method of proof does not apply for the SLP, which is open. Neither does the method apply for the WLP and the case when $n$ is even, but it is
expected that $\Bbbk[x_1,\ldots,x_n]/(x_1^2,\ldots,x_n^2)^2$ has the WLP also when $n$ is even.
\end{ex}

\section{A class of Gorenstein algebras with the SLP}
\label{sec:gor}
In this section we use \Cref{cor:symmetricHS} to detect a new class of Gorenstein algebras with the SLP. This generalizes results in \cite{gorensteinnumerical} and \cite{gorensteinglue} concerning Artinian Gorenstein algebras in codimension three, see \Cref{corollary:gorenst_numerical} and \Cref{rmk:gorenst_glue}. 

The first step is to prove that the quotient by the initial ideal, under a certain monomial ordering, has the SLP.

\begin{lemma}\label{lemma:gorenst_init}
Let $d_1, \ldots, d_n$ be nonnegative  integers, and let $\alpha_1, \ldots, \alpha_n$ be integers such that $0 \le \alpha_i \le d_i$ and $d_1=\alpha_2 + \dots + \alpha_n$. Define the monomial ideal 
	\[	K=(x_1^{d_1}, \ldots, x_n^{d_n}) + x_1^{d_1-\alpha_1}(x_2^{d_2-\alpha_2}, \ldots, x_n^{d_n-\alpha_n}).\]
	Then $R/K$ has the SLP in the narrow sense, with socle degree 
	 $d_1+ \dots + d_n-\alpha_1-n$, unless $K=R$.
\end{lemma}
\begin{proof}
Assume $K\not=R$. In particular, all $d_i$ are positive. We start by considering the following special cases.
\begin{itemize}
\item $\alpha_1=d_1$ and $\alpha_i=d_i$ for some $i>1$. In this case $K=R$.
\item $\alpha_1<d_1$, $\alpha_i=d_i$ for some $i>1$. In this case $K=(x_1^{d_1-\alpha_1},x_2^{d_2},\ldots, x_n^{d_n})$. Then $R/K$ is a monomial complete intersection, which always have the SLP, and the socle degree is $d_1+\cdots+d_n-\alpha_1-n$, as desired.
\item $\alpha_1=0$. In this case $K=(x_1^{d_1},\ldots, x_n^{d_n})$, so $R/K$ is a monomial complete intersection with socle degree $d_1+\cdots+d_n-n$. This is the desired socle degree as $\alpha_1=0$. 
\item $\alpha_1=d_1$ and $\alpha_i<d_i$ for all $i>1$. Here $K=(x_1^{d_1}, x_2^{d_2-\alpha_2}\ldots, x_n^{d_n-\alpha_n})$, so $R/K$ has the SLP and the socle degree is 
$$d_1+\cdots+d_n-\alpha_2-\cdots-\alpha_n-n=d_1+\cdots+d_n-d_1-n=d_1+\cdots+d_n-\alpha_1-n.$$ 
\end{itemize}
Assume none of the above holds. Then take $m=x_1^{d_1-\alpha_1}$ and let $I=K+(m) = (x_1^{d_1-\alpha_1}, x_2^{d_2}, \ldots, x_n^{d_n})$ 
and $J=K:(m) = (x_1^{\alpha_1}, x_2^{d_2-\alpha_2}, \ldots, x_n^{d_n-\alpha_n})$.
Both $I$ and $J$ are different from $R$. In order to apply \Cref{cor:symmetricHS}, we need to show that 
$$
(d_1+\cdots+d_n-\alpha_1-n)-(\alpha_1+d_2+\cdots+d_n-\alpha_2-\cdots-\alpha_n-n)=2(d_1-\alpha_1). 
$$
This simplifies to $d_1=\alpha_2+\cdots+\alpha_n$, which is true by assumption.
\end{proof}

Next we consider a class of ideals $K'$ such that $\ini_\prec(K')$, w.\,r.\,t.\ certain monomials orders, is the ideal $K$ from Lemma \ref{lemma:gorenst_init}. We will also consider a generalized version of the pairs $K, K'$ in Theorem \ref{thm:table_binomial}. 

\begin{theorem}\label{thm:gorenstein_slp}
Let $d_1, \ldots, d_n$ be nonnegative integers, and let $\alpha_1, \ldots, \alpha_n$ be integers such that $0 \leq \alpha_i \leq d_i$ and $d_1=\alpha_2 + \dots + \alpha_n$. Define the ideal 
	\[	K'=(x_1^{d_1}+cx_2^{\alpha_2}\cdots x_n^{\alpha_n}, x_2^{d_2}, \ldots, x_n^{d_n}) + x_1^{d_1-\alpha_1}(x_2^{d_2-\alpha_2}, \ldots, x_n^{d_n-\alpha_n}),\]
	where $c$ is any nonzero constant. Then $R/K'$ is Gorenstein and has the SLP, unless $K' = R$.
\end{theorem}
\begin{proof}
Let
	\[	\mathfrak{a}=(x_1^{d_1}+cx_2^{\alpha_2}\cdots x_n^{\alpha_n}, x_2^{d_2}, \ldots, x_n^{d_n}).\]
	As a first step we shall prove that $K'=\mathfrak{a}:(x_1^{\alpha_1})$. In other words we want to determine for which polynomials $p$ it holds that $px_1^{\alpha_1} \in \mathfrak{a}$. Let $\prec$ be any monomial order with $x_1 \succ x_i$ for all $i\ge 2$. The given generators of $\mathfrak{a}$ are a Gröbner basis w.\,r.\,t.\ $\prec$, since the leading terms are pure powers of the variables. 	 Therefore one way of proving that an element belongs to $\mathfrak{a}$ is to prove that it reduces to zero by the division algorithm on the generators of $\mathfrak{a}.$ In this case the reduction algorithm can be described as substituting $x_1^{d_1} \mapsto -cx_2^{\alpha_2} \cdots x_n^{\alpha_n}$, and removing any term divisible by $x_i^{d_i}$ for $i \ge 2$. Let $\bar{p}$ be the remainder of $p$ after applying this reduction algorithm. Then the exponent of $x_i$ in any term of $\bar{p}$ is less than $d_i$. To determine when $\bar{p}x_1^{\alpha_1}$ belongs to $\mathfrak{a}$ we shall apply the reduction once more to $\bar{p}x_1^{\alpha_1}$. We claim that no two terms in the support of $\bar{p}x_1^{\alpha_1}$ can be mapped to constant multiples of the same monomial by the substitution $x_1^{d_1} \mapsto -cx_2^{\alpha_2} \cdots x_n^{\alpha_n}$. We can ignore coefficients when proving this claim. Suppose $x_1^{\alpha_1+\beta_1}x_2^{\beta_2} \cdots x_n^{\beta_n}$ and $x_1^{\alpha_1+\gamma_1}x_2^{\gamma_2} \cdots x_n^{\gamma_n}$, with $0 \le \gamma_i, \beta_i <d_i$ for each $i$, were mapped to the same monomial. That is, we would have 
	\[
	x_1^{\alpha_1+\beta_1-kd_1}x_2^{\beta_2+k\alpha_2} \cdots x_n^{\beta_n+k\alpha_n} = x_1^{\alpha_1+\gamma_1-k'd_1}x_2^{\gamma_2+k'\alpha_2} \cdots x_n^{\gamma_n+k'\alpha_n}.
	\]
	Comparing the exponents of $x_1$ we see that $|\beta_1-\gamma_1|=|k-k'|d_1$. But $|\beta_1-\gamma_1|<d_1$ so the only option is $k=k'$ and $\beta_1=\gamma_1$. Comparing the exponents of the other variables results in $\beta_i=\gamma_i$ for all $i$, meaning that the two monomials were in fact equal from the beginning. This means that there is no cancellation of terms when we apply the substitution to $\bar{p}x_1^{\alpha_1}$. So for $\bar{p}x_1^{\alpha_1}$ to reduce to zero we need each term to be mapped to a term divisible by $x_i^{d_i}$ for some $i\ge2$. This requires each term of $\bar{p}x_1^{\alpha_1}$ to be divisible by $x_1^{d_1}$ and $x_i^{d_i-\alpha_i}$ for some $i \ge 2$. In other words $\bar{p}x_1^{\alpha_1}$ belongs to the ideal $x_1^{d_1}(x_2^{d_2-\alpha_2}, \ldots, x_n^{d_n-\alpha_n})$. Then $\bar{p} \in x_1^{d_1-\alpha_1}(x_2^{d_2-\alpha_2}, \ldots, x_n^{d_n-\alpha_n})$ and finally $p \in \mathfrak{a} + x_1^{d_1-\alpha_1}(x_2^{d_2-\alpha_2}, \ldots, x_n^{d_n-\alpha_n})=K'$. This proves that $K'=\mathfrak{a}:(x_1^{\alpha_1})$.
	
	The ideal $\mathfrak{a}$ is a complete intersection, so it follows that $R/K'$ is Gorenstein. As $K'$ is generated by monomials and one binomial it is straightforward to verify that all S-polynomials reduce to zero, w.\,r.\,t.\ $\prec$, and therefore the given generating set is a Gröbner basis. Then
	\[	\ini_\prec(K')=(x_1^{d_1}, \ldots, x_n^{d_n}) + x_1^{d_1-\alpha_1}(x_2^{d_2-\alpha_2}, \ldots, x_n^{d_n-\alpha_n}),\]
	so we know from \Cref{lemma:gorenst_init} that $R/\ini_\prec(K')$ has the SLP. It follows that $R/K'$ has the SLP as well, using Wiebe's result \cite{ini}. 
\end{proof}


We can now present a substantial generalization of the main results 
in \cite{gorensteinnumerical, gorensteinnumerical2} on the WLP for some specific classes of codimension three Gorenstein algebras arising from the Ap\'ery set of a numerical semigroup. By \cite[Theorem 5.6]{guerrieri}, all such algebras are of the form described in \Cref{thm:gorenstein_slp}, with $n=3$ and $c=-1$.

\begin{corollary}\label{corollary:gorenst_numerical}
Codimension three Gorenstein algebras arising from the Ap\'ery set of a numerical semigroup have the SLP.
\end{corollary}

\begin{rmk}\label{rmk:gorenst_glue}
	In \cite{gorensteinglue}, the authors study Artinian Gorenstein algebras obtained by 
	a topological decomposition technique, referred to as the \emph{connected sum construction}. One result, see \cite[Corollary 5.14]{gorensteinglue}, is that an algebra defined by an ideal
	\[
	(x^{a+b-k}+y^{b-k}z^a,y^{b+1},z^{a+1})+x(y^{k+1},z)
	\]
	has the SLP. This now also follows as a special case of \Cref{thm:gorenstein_slp} with
	\begin{align*}
		&d_1=a+b-k, \ d_2=b+1, \ d_3=a+1, \\
		& \alpha_1= a+b-k-1, \ \alpha_2=b-k, \ \alpha_3=a, \ \mbox{and} \ c=1.
	\end{align*}
	
	A natural question is whether the decomposition method that we propose in this paper is related to the connected sum construction.

\end{rmk}
\section{Ideals coming from tables}
\label{sec:table}

In this section we introduce a class of ideals extending those in \Cref{sec:gor}. We prove that these algebras have the SLP in the narrow sense using \Cref{cor:symmetricHS}.

\begin{definition}
An $(s,n)$-\emph{table}, where $0\le s<n$, is an $(s+1)\times n$ matrix of non-negative integers as in \Cref{cond} such that
\begin{enumerate}
\item $\alpha_{i,j}=0$ for all $i>j$,
\item $\sum_{i=1}^{s}{\alpha_{i,j}}\le d_j$ for all $j$,
\item $d_k=\sum_{i=1}^{k-1}\alpha_{i,k}+\sum_{j=k+1}^n\alpha_{k,j}+\alpha_{k+1,k+1}$ for $1 \le k \le s$, where we set $\alpha_{s+1,s+1}=0$.

\end{enumerate}
\end{definition}
  
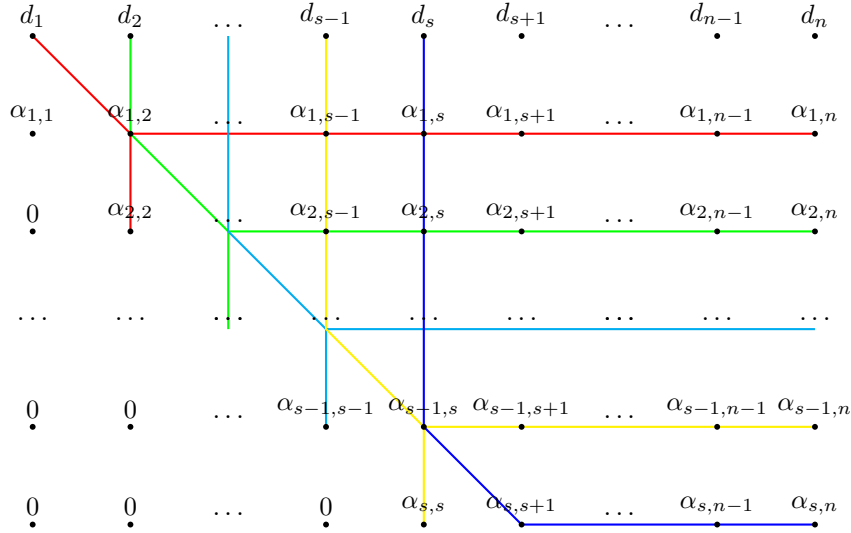
\begin{figure}[H]

\centering
\begin{tikzpicture}[scale=1.3]

      \draw [red,thick]     (2,-1) -- (9,-1);
      \draw [red,thick]    (2,-1) -- (2,-2);
      \draw [red,thick]  (2,-1)--(1,0);
      
      \draw[green, thick]  (2,-1)-- (3,-2) -- (9,-2);
      \draw[green, thick]  (3,-2) -- (3,-3);
      \draw[green,thick] (2,-1)--(2,0);

      \draw[cyan, thick]  (3,-1)-- (3,-2) -- (4,-3)--(9,-3);
      \draw[cyan, thick]  (4,-4)-- (4,-3);
      \draw[cyan, thick] (3,-1)--(3,0);

      \draw[yellow,thick]  (4,-1)-- (4,-3) --(5,-4)-- (9,-4);
      \draw[yellow,thick]  (5,-4)-- (5,-5);
      \draw[yellow,thick]  (4,-1)-- (4,0);
      \draw[blue,thick]  (5,-1)--(5,-4) --(6,-5)-- (9,-5);
      \draw[blue,thick]  (5,-1)--(5,-0);

    \foreach \x in {1,...,2}{
    \node [above, thin] at (\x,0) {$d_\x$};
    \node [above, thin] at (\x,-1) {$\alpha_{1,\x}$};
    }
    \node [above, thin] at (1,-2) {$0$};
    \node [above, thin] at (2,-2) {$\alpha_{2,2}$};

    \node [above, thin] at (4,-1) {$\alpha_{1,s-1}$};
    \node [above, thin] at (5,-1) {$\alpha_{1,s}$};
    \node [above, thin] at (6,-1) {$\alpha_{1,s+1}$};
    \node [above, thin] at (8,-1) {$\alpha_{1,n-1}$};
    \node [above, thin] at (9,-1) {$\alpha_{1,n}$};
    
    \node [above, thin] at (4,-2) {$\alpha_{2,s-1}$};
    \node [above, thin] at (5,-2) {$\alpha_{2,s}$};
    \node [above, thin] at (6,-2) {$\alpha_{2,s+1}$};
    \node [above, thin] at (8,-2) {$\alpha_{2,n-1}$};
    \node [above, thin] at (9,-2) {$\alpha_{2,n}$};

    \node [above, thin] at (4,0) {$d_{s-1}$};    
    \node [above, thin] at (5,0) {$d_{s}$};  
    \node [above, thin] at (6,0) {$d_{s+1}$};  
    \node [above, thin] at (8,0) {$d_{n-1}$};
    \node [above, thin] at (9,0) {$d_{n}$};

    \node [above, thin] at (1,-4) {$0$};
    \node [above, thin] at (2,-4) {$0$};
    
    \node [above, thin] at (4,-4) {$\alpha_{s-1,s-1}$};
    
    \node [above, thin] at (5,-4) {$\alpha_{s-1,s}$};
    
    \node [above, thin] at (6,-4) {$\alpha_{s-1,s+1}$};
    \node [above, thin] at (8,-4) {$\alpha_{s-1,n-1}$};
    \node [above, thin] at (9,-4) {$\alpha_{s-1,n}$};

    \node [above, thin] at (1,-5) {$0$};
    \node [above, thin] at (2,-5) {$0$};
    \node [above, thin] at (4,-5) {$0$};
    \node [above, thin] at (5,-5) {$\alpha_{s,s}$};
    \node [above, thin] at (6,-5) {$\alpha_{s,s+1}$};
    \node [above, thin] at (8,-5) {$\alpha_{s,n-1}$};
    \node [above, thin] at (9,-5) {$\alpha_{s,n}$};
    
    \foreach \x in {1,...,9}{
    \node [above, thin] at (\x,-3) {$\ldots$};
    }
    \foreach \y in {-5,...,0}{
     \node [above, thin] at (3,\y) {$\ldots$};
     \node [above, thin] at (7,\y) {$\ldots$};
    }

    \foreach \x in {1,...,2}{
    \foreach \y in {-2,...,0}{
    \fill[fill=black] (\x,\y) circle (0.03 cm);
    }}
    
    \foreach \x in {1,...,2}{
    \foreach \y in {-5,...,-4}{
    \fill[fill=black] (\x,\y) circle (0.03 cm);
    }}
    
    \foreach \x in {4,...,6}{
    \foreach \y in {-2,...,0}{
    \fill[fill=black] (\x,\y) circle (0.03 cm);
    }}
    \foreach \x in {4,...,6}{
    \foreach \y in {-5,...,-4}{
    \fill[fill=black] (\x,\y) circle (0.03 cm);
    }}
    
    \foreach \x in {8,...,9}{
    \foreach \y in {-2,...,0}{
    \fill[fill=black] (\x,\y) circle (0.03 cm);
    }}
    \foreach \x in {8,...,9}{
    \foreach \y in {-5,...,-4}{
    \fill[fill=black] (\x,\y) circle (0.03 cm);
    }}

\end{tikzpicture}
\caption{The $\alpha_{i,j}$'s connected by edges of the same color sum to the corresponding $d_k$.}
\label{cond}
\end{figure}

To each table $T$ we will associate a monomial ideal in $\mathbb{K}[x_1,\ldots, x_n]$
$$K(T):=K_0(T)+K_1(T)+ \cdots +K_s(T),$$
where 
$$K_0(T):=(x_1^{d_1},x_2^{d_2},\ldots,x_n^{d_n}),$$
$$K_1(T):=x_1^{d_1-\alpha_{1,1}}(x_2^{d_2-\alpha_{1,2}},\ldots,x_n^{d_n-\alpha_{1,n}}),$$
$$K_2(T):=x_1^{d_1-\alpha_{1,1}}x_2^{d_2-\alpha_{1,2}-\alpha_{2,2}}(x_3^{d_3-\alpha_{1,3}-\alpha_{2,3}},\ldots,x_n^{d_n-\alpha_{1,n}-\alpha_{2,n}}),$$
$$\vdots$$
$$K_s(T)\!:=\!x_1^{d_1-\alpha_{1,1}}\!\!\!\cdots\!  x_s^{d_s-\alpha_{1,s}- . . . -\alpha_{s,s}}(x_{s+1}^{d_{s+1}-\alpha_{1,s+1}- . . . -\alpha_{s,s+1}}\!\!\!, . . . ,x_n^{d_n-\alpha_{1,n}- . . . -\alpha_{s,n}}).$$


\begin{ex}($1$ condition, $3$ variables)
\label{allowzeros}
\begin{figure}[H]
\centering
\begin{tikzpicture}[scale=1.5]
\draw[red,thick]  (2,-1)--(3,-1);
\draw[red,thick] (2,-1)--(1,0); 
\node [above, thin] at (1,0) {$6$};
\node [above, thin] at (2,0) {$7$};
\node [above, thin] at (3,0) {$4$};
\node [above, thin] at (1,-1) {$2$};
\node [above, thin] at (2,-1) {$6$};
\node [above, thin] at (3,-1) {$0$};
\foreach \x in {1,...,3}{
    \foreach \y in {-1,...,0}{
    \fill[fill=black] (\x,\y) circle (0.03 cm);
    }}
 
\end{tikzpicture}
\caption{$1$ condition, $3$ variables}
\label{fig:31}
\end{figure}
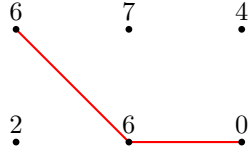

Consider the table $T$ in \Cref{fig:31}. Here we have $K_0(T)=(x^6,y^7,z^4)$, $K_1(T)=x^4(y,z^4)$. Thus $K(T)=K_0(T)+K_1(T)=(x^6,y^7,z^4,x^4y,x^4z^4)=(x^6,y^7,z^4,x^4y).$
\end{ex}

\begin{figure}[H]
\centering
\begin{tikzpicture}[scale=1]

\draw[red,thick]  (2,-2)--(2,-1)--(4,-1);
\draw[red,thick] (2,-1)--(1,0);  

\draw[green,thick]  (2,-1)--(3,-2)--(4,-2);
\draw[green,thick] (2,-1)--(2,0); 

\draw[red,thick]  (6,-1)--(8,-1);
\draw[red,thick] (6,-1)--(5,0); 

\draw[dashed] (4.5,0.5)-- (4.5,-2.5);
\draw[dashed] (8.5,0.5)-- (8.5,-2.5);

\node [above, thin] at (1,0) {$12$};
\node [above, thin] at (2,0) {$7$};
\node [above, thin] at (3,0) {$5$};
\node [above, thin] at (4,0) {$4$};
\node [above, thin] at (1,-1) {$3$};
\node [above, thin] at (2,-1) {$4$};
\node [above, thin] at (3,-1) {$3$};
\node [above, thin] at (4,-1) {$2$};
\node [above, thin] at (1,-2) {$0$};
\node [above, thin] at (2,-2) {$3$};
\node [above, thin] at (3,-2) {$2$};
\node [above, thin] at (4,-2) {$1$};

\node [above, thin] at (5,0) {$12$};
\node [above, thin] at (6,0) {$7$};
\node [above, thin] at (7,0) {$5$};
\node [above, thin] at (8,0) {$4$};
\node [above, thin] at (5,-1) {$3$};
\node [above, thin] at (6,-1) {$7$};
\node [above, thin] at (7,-1) {$3$};
\node [above, thin] at (8,-1) {$2$};

\node [above, thin] at (9,0) {$9$};
\node [above, thin] at (10,0) {$7$};
\node [above, thin] at (11,0) {$5$};
\node [above, thin] at (12,0) {$4$};

    \foreach \x in {1,...,4}{
    \foreach \y in {-2,...,0}{
    \fill[fill=black] (\x,\y) circle (0.03 cm);
    }}
    \foreach \x in {5,...,8}{
    \foreach \y in {-1,...,0}{
    \fill[fill=black] (\x,\y) circle (0.03 cm);
    }}
    \foreach \x in {9,...,12}{
    \foreach \y in {0,...,0}{
    \fill[fill=black] (\x,\y) circle (0.03 cm);
    }}

\end{tikzpicture}
\caption{Three different tables representing the same ideal $K=(x^9,y^7,z^5,w^4).$}
\label{difftables}
\end{figure}
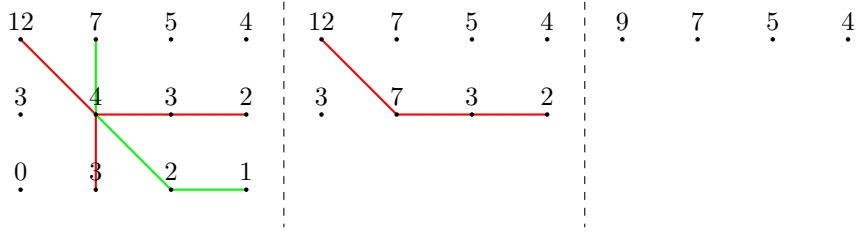

Note that if $s=0$ (that is, the table has $1$ row), then the corresponding ideal is a monomial complete intersection or even the whole ring. Therefore, in this section we will always assume $s\ge 1$. Note, however, that even if $s\ge 1$, the corresponding ideal might still be a complete intersection (as \Cref{difftables} shows) or even the whole ring.
\begin{theorem}
Let $K=K(T)$, where $T$ is a table with $s\ge 1$. Then $R/K$ has the SLP in the narrow sense, and its maximal socle degree is $d_1+ \cdots +d_n-\alpha_{1,1}-n$, unless $K=R$.
\end{theorem}
\begin{proof}
Assume $K\not=R$. We will prove the statement by induction on the number of rows of the table. The base case $s=1$ has been proven in \Cref{{lemma:gorenst_init}}. 
Let $m=x_1^{d_1-\alpha_{1,1}}$ and let $I=K+(m)$ and $J=K:(m)$. We will consider $4$ different cases:
\begin{itemize}
\item $I=J=R$. We assumed $K\not=R$. Now, $I=K+(m)=R$, therefore, $(m)=R$. But then $R=J=K:(m)=K:R=K$, which implies $K=R$ to start with, which is a contradiction.
\item $I\not=R$, $J=R$. Since $J=K:(m)=R$, we conclude that $m\in K$, which means $R\not=I=K+(m)=K$. Note that $R\not=I=(x_1^{d_1-\alpha_{1,1}},x_2^{d_2},\ldots,x_n^{d_n})$, which implies that $R/I=R/K$ has the SLP in the narrow sense and the maximal socle degree $d_1+ \cdots +d_n-\alpha_{1,1}-n$.
\item $I=R$, $J\not=R$. We assumed $K\not=R$. Now, $I=K+(m)=R$, therefore, $(m)=R$ and $R\not=J=K:(m)=K:R=K$. Note that $(m)=R$ is equivalent to $d_1=\alpha_{1,1}$. This means, we are left to show that $R/K=R/J$ possesses the SLP in the narrow sense and has the maximal socle degree $d_1+ \cdots +d_n-\alpha_{1,1}-n=d_1+ \cdots +d_n-\alpha_{1,1}-n-2(d_1-\alpha_{1,1})$. We shall return to this case later.

\item $I\not=R$, $J\not=R$. Note that $I=(x_1^{d_1-\alpha_{1,1}},x_2^{d_2},\ldots,x_n^{d_n})$. Then algebra $R/I$ possesses the SLP in the narrow sense and has the maximal socle degree $d_1+ \cdots +d_n-\alpha_{1,1}-n$. In order to apply \Cref{cor:symmetricHS}, we need to show that $R/J$ possesses the SLP in the narrow sense and has the maximal socle degree $d_1+ \cdots +d_n-\alpha_{1,1}-n-2(d_1-\alpha_{1,1})$.
\end{itemize}
To combine the last $2$ cases: we need to show that if $J\not=R$, $R/J$ has the SLP in the narrow sense and the maximal socle degree $d_1+ \cdots +d_n-\alpha_{1,1}-n-2(d_1-\alpha_{1,1})$.

Note that
\begin{align*}
J=(&x_1^{\alpha_{1,1}},x_2^{(d_2-\alpha_{1,2})}, \ldots, x_n^{(d_n-\alpha_{1,n})},\\
& x_2^{(d_2-\alpha_{1,2})-\alpha_{2,2}}(x_3^{(d_3-\alpha_{1,3})-\alpha_{2,3}},\ldots,x_n^{(d_n-a_{1,n})-a_{2,n}}),\ldots \\
& x_2^{(d_2-\alpha_{1,2})-\alpha_{2,2}}\cdots x_s^{(d_s-\alpha_{1,s})- \cdots - \alpha_{s,s}}(x_{s+1}^{(d_{s+1}-\alpha_{1,s+1})- \cdots - \alpha_{s,s+1}},\ldots\\
& \hspace{7cm} \ldots, x_n^{(d_n-\alpha_{1,n})- \cdots - \alpha_{s,n}})). 
\end{align*}
 
We also remark that 
\[\Bbbk[x_1,\ldots, x_n]/J=\Bbbk[x_2,\ldots, x_n]/J'\otimes_\Bbbk \Bbbk[x_1]/(x_1^{\alpha_{1,1}})\] 
where $J'$ is obtained from $J$ by removing the first generator. Note that $J'$ is not the whole ring and that $\alpha_{1,1}\not=0$.
By \cite[Theorem 3.34]{book}, it is enough to show that $\Bbbk[x_2,\ldots, x_n]/J'$ possesses the SLP in the narrow sense and has the maximal socle degree \[d_1+ \cdots +d_n-\alpha_{1,1}-n-2(d_1-\alpha_{1,1})-(\alpha_{1,1}-1)=-d_1+d_2+ \cdots +d_n-n+1.\] First of all note that $J'$ comes from a table. To obtain a table for $J'$, one should remove the first column of the table for $K$ together with the first condition and subtract the first two rows, as shown in \Cref{collapse}. By the induction hypothesis (noting that $J'\not=R$) we know that $J'$ possesses the SLP in the narrow sense and has the maximal socle degree $(d_2-\alpha_{1,2})+ \cdots +(d_n-\alpha_{1,n})-\alpha_{2,2}-(n-1)$. Thus we are left to show that $$(d_2-\alpha_{1,2})+ \cdots +(d_n-\alpha_{1,n})-\alpha_{2,2}-(n-1)=-d_1+d_2+ \cdots +d_n.$$ This is equivalent to $\alpha_{1,2}+\alpha_{1,3}+ \cdots +\alpha_{1,n}+\alpha_{2,2}=d_1,$ which is exactly the first condition of the table for $K$ (the removed red condition). 
\end{proof}

\begin{figure}[H]

\centering
\begin{tikzpicture}[scale=1.3]
	 \draw [red,thick]     (2,-1) -- (9,-1);
      \draw[red,thick]    (2,-1) -- (2,-2);
      \draw [red,thick]  (2,-1)--(1,0);
      
      \draw[green, thick]  (2,-1)-- (3,-2) -- (9,-2);
      \draw[green, thick]  (3,-2) -- (3,-3);
      \draw[green,thick] (2,-1)--(2,0);

       \draw[cyan, thick]  (3,-1)-- (3,-2) -- (4,-3)--(9,-3);
      \draw[cyan, thick]  (4,-4)-- (4,-3);
      \draw[cyan, thick] (3,-1)--(3,0);

      \draw[yellow,thick]  (4,-1)-- (4,-3) --(5,-4)-- (9,-4);
      \draw[yellow,thick]  (5,-4)-- (5,-5);
      \draw[yellow,thick]  (4,-1)-- (4,0);
      
      \draw[blue,thick]  (5,-1)--(5,-4) --(6,-5)-- (9,-5);
      \draw[blue,thick]  (5,-1)--(5,-0);
      \draw[dashed]  (1.5,0.5)-- (1.5,-5.5);
      \foreach \x in {2,...,9}{    
   
     \draw[dashed, -latex] (\x,-1) arc (225:135:cos 45);
     }
    \foreach \x in {1,...,2}{
    \node [above, thin] at (\x,0) {$d_\x$};

    }
    \node [above, thin] at (2,-1) {$\alpha_{1,2}$};
    \node [above, thin] at (1,-1) {$\alpha_{1,1}$};
    \node [above, thin] at (1,-2) {$0$};
    \node [above, thin] at (2,-2) {$\alpha_{2,2}$};
    \node at (1.5,-5.5) {\rotatebox{90}{\ding{34}}};

    \node [above, thin] at (4,-1) {$\alpha_{1,s-1}$};
    \node [above, thin] at (5,-1) {$\alpha_{1,s}$};
    \node [above, thin] at (6,-1) {$\alpha_{1,s+1}$};
    \node [above, thin] at (8,-1) {$\alpha_{1,n-1}$};
    \node [above, thin] at (9,-1) {$\alpha_{1,n}$};
    
    \node [above, thin] at (4,-2) {$\alpha_{2,s-1}$};
    \node [above, thin] at (5,-2) {$\alpha_{2,s}$};
    \node [above, thin] at (6,-2) {$\alpha_{2,s+1}$};
    \node [above, thin] at (8,-2) {$\alpha_{2,n-1}$};
    \node [above, thin] at (9,-2) {$\alpha_{2,n}$};

    \node [above, thin] at (4,0) {$d_{s-1}$};    
    \node [above, thin] at (5,0) {$d_{s}$};  
    \node [above, thin] at (6,0) {$d_{s+1}$};  
    \node [above, thin] at (8,0) {$d_{n-1}$};
    \node [above, thin] at (9,0) {$d_{n}$};

    \node [above, thin] at (1,-4) {$0$};
    \node [above, thin] at (2,-4) {$0$};
    
    \node [above, thin] at (4,-4) {$\alpha_{s-1,s-1}$};
    
    \node [above, thin] at (5,-4) {$\alpha_{s-1,s}$};
    
    \node [above, thin] at (6,-4) {$\alpha_{s-1,s+1}$};
    \node [above, thin] at (8,-4) {$\alpha_{s-1,n-1}$};
    \node [above, thin] at (9,-4) {$\alpha_{s-1,n}$};

    \node [above, thin] at (1,-5) {$0$};
    \node [above, thin] at (2,-5) {$0$};
    \node [above, thin] at (4,-5) {$0$};
    \node [above, thin] at (5,-5) {$\alpha_{s,s}$};
    \node [above, thin] at (6,-5) {$\alpha_{s,s+1}$};
    \node [above, thin] at (8,-5) {$\alpha_{s,n-1}$};
    \node [above, thin] at (9,-5) {$\alpha_{s,n}$};
    
    \foreach \x in {1,...,9}{
    \node [above, thin] at (\x,-3) {$\ldots$};
    }
    \foreach \y in {-5,...,0}{
     \node [above, thin] at (3,\y) {$\ldots$};
     \node [above, thin] at (7,\y) {$\ldots$};
      }
      
       \node[cross,thin, scale=0.8] at (2,-1.5) {};
       \foreach \x in {2,...,8}{
       \node[cross,thin, scale=0.8] at (0.5+\x,-1) {};
       }
     
    \foreach \x in {1,...,2}{
    \foreach \y in {-2,...,0}{
    \fill[fill=black] (\x,\y) circle (0.03 cm);
    }}
    
    \foreach \x in {1,...,2}{
    \foreach \y in {-5,...,-4}{
    \fill[fill=black] (\x,\y) circle (0.03 cm);
    }}
    
    \foreach \x in {4,...,6}{
    \foreach \y in {-2,...,0}{
    \fill[fill=black] (\x,\y) circle (0.03 cm);
    }}
    \foreach \x in {4,...,6}{
    \foreach \y in {-5,...,-4}{
    \fill[fill=black] (\x,\y) circle (0.03 cm);
    }}
    
    \foreach \x in {8,...,9}{
    \foreach \y in {-2,...,0}{
    \fill[fill=black] (\x,\y) circle (0.03 cm);
    }}
    \foreach \x in {8,...,9}{
    \foreach \y in {-5,...,-4}{
    \fill[fill=black] (\x,\y) circle (0.03 cm);
    }}

\foreach \x in {2,...,9}{    
   
  \node [thin] at (\x-0.3,-0.3) {$_{-}$};
    }

\end{tikzpicture}
\caption{Obtainiting a table for $J'$ from a table for $K$} 
\label{collapse}
\end{figure}
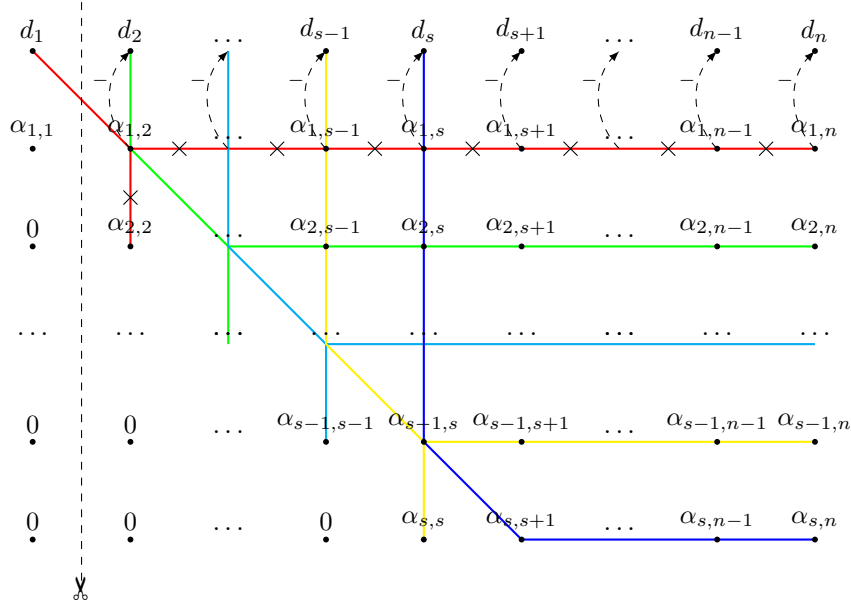

\begin{theorem}\label{thm:table_binomial}
Let $T$ be a table, $K=K(T)$ and let $K'$ be the ideal, obtained by replacing $x_1^{d_1}$ with $x_1^{d_1}+cx_2^{\alpha_{2}}x_3^{\alpha_{3}}\cdots x_n^{\alpha_{n}}$, where $c$ is any constant and $\alpha_i$ is the sum of all numbers of column $i$ which contribute to the first (red) condition. In other words, $\alpha_2=\alpha_{1,2}$ for $s=1$, $\alpha_2=\alpha_{1,2}+\alpha_{2,2}$ for $s\ge 2$, and $\alpha_i=\alpha_{1,i}$ for $i\ge 3$ and for any $s$. Then $R/K'$ has the same Hilbert series as $R/K$ and possesses the SLP.
\end{theorem}
\begin{proof}
Similarly to \Cref{thm:gorenstein_slp}, the initial ideal of $K'$ equals $K$ with respect to a monomial order fulfilling $x_1 \succ x_i$. 
Therefore, by Wiebe's result, we are done.
\end{proof}
\begin{rmk}
Let $K$ and $K'$ be as above, and set $s=1$. Then we will recover the classes of ideals from \Cref{sec:gor}. However, if $s>1$, the corresponding algebras will not be Gorenstein in general.
\end{rmk}

\begin{question} \label{q:main}
All the ideals in \Cref{sec:gor} and \Cref{sec:table}, as well as the class of ideals in \Cref{ex:1}, are glued centre-to-centre of several CI's, therefore they have the SLP in the narrow sense. Can this method be used to detect other classes of monomial algebras with SLP in the narrow sense? One should note that not all monomial algebras with the SLP in the narrow sense can be obtained in this way, see Example \ref{ex:not_glued}.
\end{question}

\begin{ex}\label{ex:not_glued}
Let $R=\Bbbk[x,y,z]$ and $K=(x^3,y^3,z^5,x^2y^2,xz,yz)$. The ideal $K$ is the initial ideal, w.\,r.\,t.\ the Lex order, of the ideal defined by the Macaulay dual generator $X^2Y^2+Z^4$. The algebra $R/K$ has the SLP and the Hilbert series $1+3t+4t^2+3t^3+t^4$. However, there is no monomial $m$ such that the conditions on the Hilbert functions of $I=K+(m)$ and $J=K:(m)$ in \Cref{thm:glue_maxrank} hold for all $d$.
\end{ex}

\section{Monomial almost complete intersections}
\label{sec:maci}

Stanley's result \cite{stanley}, that monomial complete intersections have the SLP, mentioned in Example \ref{ex:1}, is usually seen as the starting point of the study of the Lefschetz properties. Therefore it is natural to consider the Lefschetz properties for monomial almost complete intersections (m.a.c.i.'s), that is, algebras defined by ideals of the form 
$(x_1^{d_1}, \ldots, x_n^{d_n}, m)$, where $m$ is a monomial. 

 Brenner and Kaid \cite{brennerkaid} gave the now famous example that the algebra generated by $x_1^3,x_2^3,x_3^3,x_1 x_2 x_3$ fails the WLP. Migliore, Miro-Ro\'ig, and Nagel \cite{MMN} generalized this result and showed that $x_1^d,x_2^d,x_d^d,x_1 \cdots x_d$ fails the WLP for all $d \geq 3$. They also undertook a deep study of codimension three \emph{level} m.\,a.\,c.\,i.'s, establishing a conjecture \cite[Conjecture 6.8]{MMN} of the presence of this WLP for this class and proved the sufficient part of the conjecture.

Cook II and Nagel \cite{lozenges, maci} used the connection between monomial ideals in three variables and bipartite matchings to give further evidence for the conjecture. 
 
Here we focus on the equigenerated case. We will present two conjectures on equigenerated m.\,a.\,c.\,i.'s, one for the SLP, and one for the WLP. 

Our conjectures suggest a classification of the SLP and the WLP in the codimension three case, and we give partial necessary conditions in the SLP case using our basis splitting argument, and sufficient conditions using  the representation theory of $S_2$ in the WLP case.

\subsection{The SLP for some equigenerated m.\,a.\,c.\,i.'s}
\begin{theorem}
	The algebra defined by the ideal $(x_1^a,x_2^a, \ldots, x_n^a,x_1^{a-1}x_2)$ has the SLP.
\end{theorem}
\begin{proof}
	Let $K=(x_1^a,x_2^a, \ldots, x_n^a,x_1^{a-1}x_2)$, 
	\[
	I=K+(x_1^{a-1})=(x_1^{a-1},x_2^a, \ldots, x_n^a), \ \mbox{and}
	\]
	\[
	J=K:(x_1^{a-1})=(x_1,x_2,x_3^a, \ldots, x_n^a).
	\]
	Both $R/I$ and $R/J$ are monomial complete intersections, so they both have the SLP. The maximal socle degree of $R/I$ is $n(a-1)-1$ and the maximal socle degree of $R/J$ is $(n-2)(a-1)$. We shall now check that the conditions on the Hilbert function in \Cref{thm:glue_maxrank}, with $m=x_1^{a-1}$, are satisfied. If $\HF(R/I,i)<\HF(R/I,i+d)$ then we must have $2i+d <n(a-1)-1$. This implies $2(i-a+1)+d<(n-2)(a-1)-1< (n-2)(a-1)$, which means that $\HF(R/J,i-a+1)<\HF(R/J,i-a+1+d)$. If $\HF(R/I,i)>\HF(R/I,i+d)$ then $2i+d >n(a-1)-1$, which implies $2(i-a+1)+d>(n-2)(a-1)-1$. Then $2(i-a+1)+d \ge(n-2)(a-1)$, which means that $\HF(R/J,i-a+1)\ge \HF(R/J,i-a+1+d)$. 
\end{proof}

We believe that the algebra defined by $(x_1^a,x_2^a, \ldots, x_n^a,x_1^{a-2}x_2x_3)$ for $a \ge 4$ also has the SLP, but unfortunately our basis splitting argument does not seem to apply in this case. For the codimension three case we have a conjecture for all possible m.\,a.\,c.\,i.’s with the SLP. The conjecture is based on computer calculations in \emph{Macaulay2} \cite{M2} using the package \emph{MaximalRankProperties} \cite{maxrank}.

\begin{conjecture}
The equigenerated m.\,a.\,c.\,i.’s of codimension three with the SLP are given by
$(x_1^a,x_2^a, x_3^a,x_1^{a-1}x_2)$ with $a \geq 2$, $(x_1^a,x_2^a, x_3^a,x_1^{a-2}x_2x_3)$ with $ a \geq 4$, $(x_1^5,x_2^5,x_3^5,x_1^2x_2^2x_3)$ and $(x_1^7,x_2^7,x_3^7,x_1^3x_2^2x_3^2)$.
\end{conjecture}

\subsection{The WLP for codimension three equigenerated m.a.c.i's}
We also make a conjecture for the WLP, again based on computer experiments.  Let $R_{a,b,c} = \Bbbk[x,y,z]/(x^d,y^d,z^d,x^a y^b z^c)$, where $ a \geq b \geq c$ and $d = a + b+ c$.


\begin{conjecture} \label{conjwlp}
The algebra $R_{a,b,c}$ fails the WLP if and only if the following two conditions hold:
\begin{enumerate}
\item  $d = 6 k  + 3$ for an integer $k$ such that $a < 4  k + 2$,
 \item  among the numbers $a,b$ and $c$, at least two are equal. 
 \end{enumerate}
\end{conjecture}

We can prove the sufficient part of \Cref{conjwlp}. The method of proof is inspired by the proof of \cite[Theorem 18]{powers}. 

\begin{theorem} \label{maciWLP}
Suppose the conditions in \Cref{conjwlp} hold. Then the algebra $R_{a,b,c}$
fails the WLP. 
\end{theorem}
In the proof we will use the following lemma.
\begin{lemma} \label{lemma}

Let $a \geq b \geq c$, and let $d = a +b + c$. Suppose that $d = 6k+3$ for some integer $k$ and that  
$a<4  k + 2$. 
Then the value of the Hilbert function of the algebra $R_{a,b,c}$ 
 in degrees $8k+2$ and $8k+3$ is equal to $6 (2k+1)^2$.

\end{lemma}

\begin{proof}
We have $(x^a y^b z^c) \cap (x^d) = (x^d y^b z^c)$, and $d+b+c > 6k+3 + (6k+3-(4k+2)) = 8k+4$, showing that $(x^a y^b z^c)$ and  $(x^d)$ have empty intersection in degrees less than or equal to $8k+4$. 
We also have $(x^a y^b z^c) \cap (y^d) = (x^a y^d z^c)$, and $a+d+c \geq d + b + c > 8k+4$, so $(x^a y^b z^c)$ and $(y^d)$ have empty intersection in degrees less than or equal to $8k+4$. A similar argument applies to $(x^a y^b z^c) \cap (z^d)$. Therefore, the Hilbert function in degree $i \in [6k+3,8k+4]$ is equal to $\binom{i+2}{2} - 4 \binom{i-d+2}{2}$. A calculation shows that  
$\binom{8k+2 + 2}{2} - 4 \binom{8k+2-(6k+3)+2}{2} = \binom{8k+3 + 2}{2} - 4 \binom{8k+3-(6k+3)+2}{2}  = 6 (2k+1)^2$.
\end{proof}

\begin{proof}[Proof of \Cref{maciWLP}]
Suppose that $d = 6k+3$, $a < 4k+2$ and that among the numbers $a,b$ and $c$, at least two are equal. Since $R_{a,b,c}$ is a monomial algebra, $R_{a,b,c}$ has the WLP if and only if 
 multiplication by $x+y+z$ has maximal rank in every degree. 

{\bf Case 1:} Suppose that $b=c$, implying $a$ is odd.
First of all we would like to compute the number of monomials of $R_{a,b,b}$ in degrees $8k+2$ (respectively, $8k+3$) fixed by the permutation interchanging $y$ and $z$. Let $F_{8k+2}$ (respectively, $F_{8k+3}$) denote the corresponding set of monomials. Let $x^{d_1}(yz)^{d_2}\in F_{8k+3}$. Note that $d_1$ is odd, and thus we can divide this monomial by $x$ and get a monomial in $F_{8k+2}$. That gives us a map
\begin{align*}
 F_{8k+3} &\rightarrow F_{8k+2}\\
x^{d_1}(yz)^{d_2} &\mapsto x^{d_1-1}(yz)^{d_2}.
\end{align*}

This map is clearly injective. We would like to know how much it fails to be surjective. If this map was surjective, then every monomial in $F_{8k+2}$, multiplied by $x$, would give a monomial in $F_{8k+3}$. However, this is not the case. Indeed, there are monomials $m\in F_{8k+2}$ such that $mx=0$. These are exactly monomials outside of the image of the map above. We would like to know how many such monomials there are. If $m \neq 0$ and $mx=0$, we have two options. Either the $x$-exponent of $m$ is $d-1=6k+2$, in which case $m=x^{6k+2}(yz)^k$, or the $x$-exponent of $m$ is $a-1$, in which case $m=x^{a-1}(yz)^{\frac{8k+3-a}{2}}$. Note that this makes sense since $a$ is odd. It is easy to see that in both cases indeed $mx=0$. We conclude that $|F_{8k+2}|-|F_{8k+3}|=2$.




Consider the symmetric group on two elements and let the non-identity element act by interchanging $y$ and $z$ in $R_{a,b,c}$. From \Cref{lemma} and the above calculation, we get that the multiplicity of the trivial representation in degree $8k+2$ equals 
$(6 (2k+1)^2 + s)/2$, while the multiplicity in degree $8k+3$ equals 
$(6 (2k+1)^2 + (s - 2))/2$. Since these two numbers differ, the multiplication by $x+y+z$ can not be an isomorphism by Schur's lemma.

{\bf Case 2:} Suppose that $a=b$, implying $c$ odd. A similar argument as in Case 1 gives that if we denote by $s$ the number of monomials in degree $8k+2$ that are fixed by the permutation interchanging $x$ and $y$, then the number of monomials in the next degree that are fixed by the same permutation equals $s-2$. Therefore, we can use Schur's lemma again to conclude that the multiplication by $x+y+z$ can not be an isomorphism.
\end{proof}

\begin{rmk}
It is interesting to notice the similarity between \Cref{maciWLP} and \cite[Theorem 7.2]{MMN} -- in both cases the failure of the WLP occurs when the Hilbert function has what is referred to in \cite{MMN} as "twin peaks".
\end{rmk}

By using a special case of \cite[Theorem 4.10]{maci}, namely that if  $a \geq b \geq c$ and $d = a +b + c$, then 
the algebra $R_{a,b,c}$ has the WLP in all cases given by \Cref{conjwlp} except the case $d = 6k+3$ and $4k+2 > a > b > c>0,$ this gives, together with \Cref{maciWLP}, the following result. 

\begin{proposition} \label{wlpprop}
To show that \Cref{conjwlp} holds true, it is enough to show that the algebra $R_{a,b,c}$ has the WLP, where $d = 6k+3$ and $4k+2  > a > b > c>0$. 
\end{proposition}

Unfortunately, even though the case left open might at first sight look like a tractable problem, we have so far failed in our attempts in providing a proof.
\begin{rmk}
There is a misprint in \cite[Theorem 4.10]{maci}: the statement (b) (2) (II) should read 
$$a + b + c - 2 (\alpha + \beta + \gamma) \text{ is divisible by } 6 $$
instead of
$$a + b + c + \alpha + \beta + \gamma \text{ is divisible by } 6.$$
\end{rmk}

\section{An attempt of gluing ideals}
\label{sec:gluing}

Let $I^c$ denote monomials outside of $I$, as before. Let $G(I)$ denote the (unique) minimal generating set of $I$. We have seen that, given an ideal $K$ and a monomial $m$, we can define $I=K+(m)$ and $J=K:(m)$, so that $\HS(R/K,t)=\HS(R/I,t)+t^{deg(m)}\HS(R/J,t)$. In a sense, $K$ is decomposed into $I$ and $J$ with respect to $m$. We are interested in the reverse process: given $I$ and $J$ do there exist $K$ and $m$ such that $I=K+(m)$ and $J=K:(m)$? In other words, does there exist a $K$ glued from $I$ and $J$ w.\,r.\,t.\ some $m$?



If one of the ideals $I, J, (m)$ and $K$ is trivial (that is, equal to $0$ or $R$), then we will not get anything new and non-trivial. Therefore, all ideals $I$, $J$, $(m)$ and $K$ will be proper from now on.

\begin{proposition}
\label{propglue}
Let $I=I'+(m)$. Assume that $I':(m)\subseteq J$ and let $K=I'+mJ$. Then $\HS(R/K,t)=\HS(R/I,t)+t^{deg(m)}\HS(R/J,t)$.
\end{proposition}

\begin{proof}
Monomials outside of $K$ which are not divisible by $m$ are exactly the same as monomials outside of $(m)+K=(m)+I'+mJ=(m)+I'=I$.
Monomials outside of $K$ which are divisible by $m$ are multiples of $m$ by monomials outside of $K:(m)=(I'+mJ):(m)=I':(m)+mJ:(m)=I':(m)+J=J$. 
\end{proof}

\begin{rmk}
Note that from the proof of \Cref{propglue} it follows that $I=K+(m)$ and $J=K:(m)$. Thus $K$ is gluing of $I$ and $J$ with respect to $m$. 
\end{rmk}
\begin{rmk}
 Note that $I'$ is not necessarily a proper ideal. If $I'=R$, then $I=I'+(m)=R$ (thus we shall not consider this case), but we do not rule out $I'=0$. For example, let $I'=0$, $I=I'+(m)=(m)$, where $(m)$ is a proper ideal. Then $I':(m)=0:(m)=0$ (since $(m)\not=0$) and thus we can take any proper ideal $J$. Then $K=I'+mJ=mJ$ is glued from $I=(m)$ and $J$.
\end{rmk}
\begin{rmk}
If there exists a gluing of $I$ and $J$ with respect to $m$, then $m$ belongs to $G(I)$, the set of minimal generators of $I$. Indeed, if such a gluing $K$ exists, then necessarily $I=K+(m)$ and $J=K:(m)$. In particular, $I=K+(m)$ implies $m\in I$. However, if there exists $m'\in I$ such that $m'|m$, $m'\not=m$, and given that $I=K+(m)$, we get $m'\in K$. Thus $m\in K$, in which case $I=K$, $J=K:(m)=R$, and therefore, it is a trivial gluing.
\end{rmk}

Now that we know that $m\in G(I)$, we can set $I_m:=(G(I)\backslash \{m\})$.
\begin{rmk}
\label{remm}
No element of $G(I_m)=G(I)\backslash \{m\}$ is divisible by $m$ and no element of $G(I_m)$ divides $m$ since otherwise $G(I)$ was not a minimal generating set of $I$ to start with.
\end{rmk}
In the following, by 
$B\in [A,C]$ we mean  $A \subseteq B \subseteq C$ for ideals $A,B,C.$
\begin{proposition} \label{prop:nonunique}
Let $I$, $I'$, $J$, $m$, $K$ be as in \Cref{propglue}, and let $I_m$ be as defined above. Then
\begin{enumerate}
\item $I'\in [I_m,K]$, in particular, $I_m\subseteq K$,
\item any ideal $I''\in [I_m,K]$ satisfies conditions of \Cref{propglue} and gives the same gluing $K$.
\end{enumerate}
\end{proposition}

\begin{proof}
\,
\begin{enumerate}
\item The inclusion $I'\subseteq K$ is obvious since $K=I'+mJ$. Assume $I_m\not\subseteq I'$. Then there is an element $e\in I_m\backslash I'$. Note that in this case $e$ can be chosen to be in $G(I_m)$. Then $e\in I_m+(m)=I=I'+(m)$, that is, $e\in (m)$, which is a contradiction to $e\in G(I_m)$ according to \Cref{remm}. We can conclude that $I_m\subseteq I'$.
\item Let $I''\in [I_m,K]$. Then $I''+(m)\in [I_m+(m),K+(m)]=[I,I]=I$. Also, $I'':(m)\subseteq K:(m)=J$, that is, the conditions of \Cref{propglue} are satisfied. Then $I''+mJ\in[I_m+mJ,K+mJ]=[I_m+mJ,K]$. If we prove that $I_m+mJ=K$, we are done. We already know that $I_m\subseteq K$ and $mJ\subseteq K$, thus it is enough to show that $K\subseteq I_m+mJ$. Assume $e\in K\backslash I_m$. Then $e\in K+(m)=I=I_m+(m)$, which implies $e\in (m)$, say, $e=me_1$. Then $e_1\in (e):(m)\subseteq K:(m)=J$, that is, $e_1\in J$ and thus $e=me_1\in mJ$. \qedhere
\end{enumerate}
\end{proof}

\begin{ex}
\label{ex:big} Let $I'=(x^2,y^3,z^4,xy^2,xz^3,xyz)$ and $m=x$. Then $$I=(x)+I'=(x,y^3,z^4),$$
$$I':(x)=(x,\cancel{y^3},\cancel{z^4},y^2,z^3,yz)=:J,$$ and 
\begin{align*}
K=I'+xJ=&(x^2,y^3,z^4,xy^2,xz^3,xyz)+x(x,y^2,z^3,yz)\\
=&(x^2,y^3,z^4,xy^2,xz^3,xyz).
\end{align*}
Clearly $\Bbbk[x,y,z]/I$ has the SLP, and it is easily checked that so does $\Bbbk[x,y,z]/J$. Moreover, $K,m,I,J$ satisfy the assumptions of \Cref{thm:glue_maxrank} for any $d$, so it follows that $\Bbbk[x,y,z]/K$ has the SLP.
\end{ex}

\begin{rmk}
If in \Cref{ex:big} we instead let 
$I':=(y^3,z^4)$, then $$I=(x)+I'=(x,y^3,z^4),$$
$$I':(x)=(y^3,z^4)\subseteq (x,y^2,z^3,yz)=:J,$$
$$K=I'+xJ=(y^3,z^4)+x(x,y^2,z^3,yz)=(x^2,y^3,z^4,xy^2,xz^3,xyz),$$
which shows that the gluing of $I$ and $J$ with respect to $m$, if it exists, is independent of $I'$, demonstrating \Cref{prop:nonunique}. In \Cref{ex:big} we took the biggest posible choice of $I'$ which is $I'=K$. Here we are taking the smallest possible choice, $I'=I_m$.
\end{rmk}

Let $I$ be a nontrivial ideal given by its minimal generating set $G(I)$. From all of the above we see that:
\begin{enumerate}
\item Any potential monomial $m$ for gluing is contained in $G(I)$. So let us choose some $m\in G(I)$.
\item As soon as $I$ and $m\in G(I)$ are chosen, we have the following. Assume that $I'$ and $J$ are chosen to satisfy the gluing conditions and $K=I'+mJ$. Then we know that any $I''\in [I_m, K]$ would do the same gluing. Thus, when $I$ and $m$ are known, we can without loss of generality set $I'=I_m$, even though $J$ is not known yet.
\item Now we can choose any $J$ such that $I_m:(m)\subseteq J$. In particular, every element of $G(I)$ can be used for some nontrivial gluing since a nontrivial choice of $J\supseteq I_m:(m)$ will always exist. Indeed, this is only impossible if $I_m:(m)=R$, that is, $m\in I_m$, which is a contradiction according to \Cref{remm}. 
\end{enumerate}
\begin{rmk}
Note that if we choose $J=I_m:(m)$, then $K=I_m+mJ=I_m$ and the interval $[I_m, K]$ becomes a single point.
\end{rmk}
We reach the following conclusion.
\begin{theorem} The ideals
$I$ and $J$ can be glued together if and only if there exists $m\in G(I)$ such that $I_m:(m)\subseteq J$. In this case the gluing is done with respect to $m$ and $K=I_m+mJ$. 
\end{theorem}
\begin{ex} \label{ex:last}
Let $I=(x^3,y^2,z^4)$, $J=(x,y,z)$. Any monomial in $G(I)$ can be chosen to be $m$. If, say, $m=x^3$, then $I_m=(y^2,z^4)$ and since $I_m:(m)\subseteq J$, we  have $K=I_m+mJ=(y^2,z^4)+x^3(x,y,z)=(x^4,y^2,z^4,x^3y,x^3z)$. 
Note that $I$ and $J$ are monomial complete intersections; the maximal socle degree of $R/I$ is $6$, the maximal socle degree of $R/J$ is $0$, and $\deg(m)=3$. Thus $K$, $m$, $I$, $J$ satisfy the conditions of \Cref {cor:symmetricHS} and therefore $K$ is even a centre-to-centre gluing of $I$ and $J$. In particular, $R/K$ has the SLP in the narrow sense. Another way to see it is to note that $K$ is the ideal associated to the table in \Cref{gluetable}.

\begin{figure}[H]

\centering
\begin{tikzpicture}[scale=1.5]

\node [above, thin] at (1,0) {$4$};
\node [above, thin] at (2,0) {$2$};
\node [above, thin] at (3,0) {$4$};
\node [above, thin] at (1,-1) {$1$};
\node [above, thin] at (2,-1) {$1$};
\node [above, thin] at (3,-1) {$3$};
\foreach \x in {1,...,3}{
    \foreach \y in {-1,...,0}{
    \fill[fill=black] (\x,\y) circle (0.03 cm);
    }}
\draw[red,thick]  (2,-1)--(3,-1);
\draw[red,thick] (2,-1)--(1,0);  
\end{tikzpicture}
\caption{A table giving the ideal in \Cref{ex:last}.}
\label{gluetable}
\end{figure}
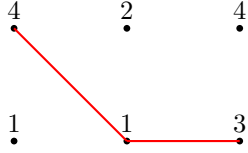 

\end{ex}

Note that if for given $I$ and $J$ there exist several suitable choices of $m$, then we will necessarily obtain different gluings. Indeed, let $K=I_{m_1}+m_1J=I_{m_2}+m_2J$, where $m_1, m_2\in G(I), m_1\not=m_2$. Then $m_1\in I_{m_2}\subseteq K$. Then $m_1\in I_{m_1}+m_1J$, which implies $J=K=R$.

\section*{Acknowledgement}
The authors are grateful to Nasrin Altafi, Chris McDaniel, Mart\'i Salat, Akihito Wachi, and Junzo Watanabe for discussions on the applicability of our monomial decomposition method during the partly virtual workshop on the Lefschetz Properties in Algebra, Geometry and Combinatorics at Oberwolfach in the fall 2020. These discussions gave us more insight of the method, and led us to Example \ref{ex:not_glued}, which provides a negative answer to a question in an earlier version of this paper. 

We also thank the referee for careful reading and valuable comments.

\end{document}